\documentclass[11pt]{amsart}

\usepackage{graphicx}
\usepackage[english]{babel}
\usepackage[T1]{fontenc}
\usepackage[latin1]{inputenc}
\usepackage{amsfonts}
\usepackage{amssymb}
\usepackage{amsthm}
\usepackage{amsmath}
\usepackage{tikz}
\usepackage{float}
\usepackage{enumerate}
\usepackage{accents}
\usepackage{mathtools}
\usepackage{mathrsfs}
\usepackage{comment}
\usepackage{afterpage}
\usepackage{bbm}
\usepackage{dsfont}
\usepackage{stmaryrd}
\usepackage{hyperref}
\usepackage{mleftright}
\usepackage{subfig}
\usepackage{soul}
\usepackage{tikz-cd} 
\usepackage{fullpage}

\theoremstyle{plain}
\newtheorem{thm}{Theorem}[section]
\newtheorem{lem}[thm]{Lemma}
\newtheorem{prop}[thm]{Proposition}
\newtheorem{cor}[thm]{Corollary}
\newtheorem*{thm*}{Theorem}
\newtheorem*{prop*}{Proposition}
\newtheorem*{cor*}{Corollary}
\newtheorem{thmintro}{Theorem}

\newtheorem{corintro}[thmintro]{Corollary}

\theoremstyle{definition}
\newtheorem{defn}[thm]{Definition}

\newtheorem{rmk}[thm]{Remark}
\newtheorem{ass}[thm]{Assumption}
\newtheorem*{rmk*}{Remark}
\newtheorem*{ex*}{Example}

\newtheorem*{quest*}{Question}

\newtheorem*{defn*}{Definition}

\newcommand{\acts}{\curvearrowright}
\newcommand{\ra}{\rightarrow}
\newcommand{\Ra}{\Rightarrow}

\newcommand{\cu}{\subseteq}

\renewcommand{\o}{\circ}

\newcommand{\mc}{\mathcal}
\newcommand{\mf}{\mathfrak}

\newcommand{\Z}{\mathbb{Z}}

\newcommand{\s}{\sigma}
\newcommand{\eps}{\epsilon}
\newcommand{\Om}{\Omega}
\newcommand{\om}{\omega}

\renewcommand{\L}{\Lambda}
\newcommand{\g}{\gamma}
\newcommand{\G}{\Gamma}

\newcommand{\CAT}{{\rm CAT(0)}}

\DeclareMathOperator{\diam}{diam}

\newcommand{\E}{\mc{E}}
\renewcommand{\H}{\mc{H}}

\begin{document}

\title{Connected components of Morse boundaries of graphs of groups} 
\author[E. Fioravanti]{Elia Fioravanti}\address{Max Planck Institute for Mathematics, Bonn, Germany}\email{fioravanti@mpim-bonn.mpg.de} 
\author[A. Karrer]{Annette Karrer}\address{Technion, Haifa, Israel}\email{annettek@campus.technion.ac.il} 

\begin{abstract}
Let a finitely generated group $G$ split as a graph of groups. If edge groups are undistorted and do not contribute to the Morse boundary $\partial_MG$, we show that every connected component of $\partial_MG$ with at least two points originates from the Morse boundary of a vertex group. 

Under stronger assumptions on the edge groups (such as wideness in the sense of Dru\c{t}u--Sapir), we show that Morse boundaries of vertex groups are topologically embedded in $\partial_MG$.
\end{abstract}

\maketitle

\section{Introduction.}

Morse boundaries $\partial_MG$ of finitely generated groups $G$ were introduced by Charney, Sultan and Cordes \cite{Charney-Sultan,Cordes} in an attempt to extend to all groups some of the most fundamental properties of Gromov boundaries of hyperbolic groups \cite{Gromov}. Importantly, every quasi-isometry between finitely generated groups extends to a homeomorphism of their Morse boundaries, which can provide a useful tool to distinguish  quasi-isometry classes of groups.

When $G$ is not Gromov-hyperbolic, the topology of $\partial_MG$ is rather unwieldy (for instance, it is not 1st countable or compact), which often makes an explicit computation difficult. Despite this, Charney, Cordes and Sisto recently showed that essentially all known examples of infinite, \emph{totally disconnected} Morse boundaries fall into just two homeomorphism classes: the Cantor space and the $\om$--Cantor space \cite{Charney-Cordes-Sisto}. 

While groups whose Morse boundary is a Cantor space are fully classified (they are hyperbolic, hence virtually free), it appears that the class of groups with $\om$--Cantor boundary is rather large. For instance, it includes all irreducible, non-free right-angled Artin groups, as well as all non-geometric graph manifold groups \cite{Charney-Cordes-Sisto}.

Our first goal is to expand the class of finitely generated groups known to have totally disconnected Morse boundary. A natural source of examples is provided by graphs of groups. We restrict to the situation where no Morse ray in $G$ stays at bounded distance from an edge group, and show that all nontrivial connected components of $\partial_MG$ originate from vertex groups (Theorem~\ref{components intro}). 

As motivation for our assumptions, recall that surface groups split as graphs of groups with cyclic edge groups and non-abelian free vertex groups. Thus, when edge groups are allowed to contain Morse rays in $G$, the Morse boundary $\partial_MG$ can be connected (here:~a circle) even if all Morse boundaries of vertex groups are totally disconnected (Cantor sets).

More generally, we investigate necessary and sufficient conditions for Morse boundaries of vertex groups to be topologically embedded in $\partial_MG$, even when they are not totally disconnected. Our main result in this direction is Theorem~\ref{embedding intro}, which applies, for instance, to all graphs of groups with (undistorted, one-ended) solvable edge groups. This is new already for JSJ decompositions of irreducible $3$--manifold groups.

\medskip
As a set, $\partial_MG$ is defined as the collection of all Morse geodesic rays in a fixed Cayley graph of $G$, identifying rays at finite Hausdorff distance. Our topology of choice will always be the \emph{direct limit} topology from \cite{Charney-Sultan,Cordes}. 

We mention that an alternative topology on $\partial_MG$ was introduced by Cashen and Mackay in \cite{Cashen-Mackay}. The latter has the advantage of being metrisable, while retaining quasi-isometric invariance. However, it appears to be more complicated to describe explicitly, and we are not aware of a single non-empty Morse boundary of a non-hyperbolic group for which the Cashen--Mackay topology can be given ``intrinsic'' characterisations in the spirit of \cite{Charney-Cordes-Sisto}.

The following relative version of Morse boundaries plays an important role in our results.

\begin{defn*}[Relative Morse boundary]
Let $H\leq G$ be finitely generated groups, with $H$ undistorted in $G$. The \emph{relative Morse boundary} of $H$ in $G$, denoted $(\partial_MH,G)$, is the subset of $\partial_MH$ consisting of points represented by rays that remain Morse in the Cayley graphs of $G$. 

We always endow $(\partial_MH,G)$ with the subspace topology coming from $\partial_MH$. {\bf This is the opposite of the convention adopted in \cite{Karrer}.}
\end{defn*}

Note that we always have a natural injection $(\partial_MH,G)\hookrightarrow\partial_MG$. Our first result greatly extends the main theorem of \cite{Karrer}, while providing a significantly simpler proof.

\begin{thmintro}\label{components intro}
Let a finitely generated group $G$ split as a graph of groups. Suppose that:
\begin{itemize}
\item all edge groups are finitely generated and undistorted in $G$;
\item $(\partial_ME,G)=\emptyset$ for every edge group $E$.
\end{itemize}
If a connected component $\mc{C}\cu\partial_MG$ is not a singleton, then $\mc{C}$ is contained in the image of the natural injection $(\partial_MV,G)\hookrightarrow\partial_MG$ for a vertex group $V\leq G$. Furthermore, if $V_1,V_2$ are distinct vertex groups, then $(\partial_MV_1,G)\cap(\partial_MV_2,G)=\emptyset$.
\end{thmintro}

The following consequence of Theorem~\ref{components intro} has some overlap with \cite[Theorem~1.2]{Charney-Cordes-Sisto}. The advantage is that here we do not require acylindricity of the splitting, nor that vertex groups have trivial Morse boundary.

\begin{corintro}\label{cor intro}
Under the assumptions of Theorem~\ref{components intro}, suppose additionally that $(\partial_MV,G)$ is totally disconnected for every vertex group $V\leq G$. Then $\partial_MG$ is totally disconnected.
\end{corintro}

We emphasise that the empty set is totally disconnected. Thus, in Corollary~\ref{cor intro}, the relative Morse boundaries $(\partial_MV,G)$ are allowed to be empty and we make no claim that $\partial_MG$ will be nonempty.

Corollary~\ref{cor intro} follows from Theorem~\ref{components intro} because, for every undistorted subgroup $H\leq G$, the natural inclusion $(\partial_MH,G)\hookrightarrow\partial_MG$ is an \emph{open} map (see Lemma~\ref{open map}).
However, even in the setting of Theorem~\ref{components intro}, the inclusions $(\partial_MV,G)\hookrightarrow\partial_MG$ need not be \emph{continuous}, as demonstrated by the following example.

\begin{ex*}
Consider the group $G=\Z^2\ast\Z=\langle x,y\rangle\ast\langle z\rangle$. It admits the splitting $G=\langle x,y\rangle\ast_{\langle y \rangle}\langle y,z\rangle$. Since $\langle y\rangle$ lies in the flat $\langle x,y\rangle$, it is undistorted and has trivial relative Morse boundary in $G$.

Consider the vertex group $V:=\langle y,z\rangle\simeq F_2$, which is also undistorted. The inclusion:
\[ (\partial_MV,G)\hookrightarrow\partial_MG \]
is not continuous. In order to see this, consider the rays labelled by $z^ny^nz^{\infty}$ and $z^{\infty}$, which all lie in $(\partial_MV,G)$. Since $V$ is hyperbolic, we have $z^ny^nz^{\infty}\ra z^{\infty}$ in the topology of $\partial_MV$. 

However, since the rays $z^ny^nz^{\infty}$ spend longer and longer in the flat $\langle x,y\rangle$, they are not uniformly Morse in $G$. It follows that they form a closed subset of $\partial_MG$ (every stratum of $\partial_MG$ contains only finitely many of them). Hence $z^ny^nz^{\infty}\not\ra z^{\infty}$ in the topology of $\partial_MG$. This last argument is taken from \cite[Section~5]{Murray}.
\end{ex*}

Of course, it would be desirable to have conditions ensuring that the injections ${(\partial_MV,G)\hookrightarrow\partial_MG}$ in Theorem~\ref{components intro} are topological embeddings, as this would be an important step towards fully characterising $\partial_MG$ in terms of boundaries of vertex groups. 

To this regard, note that a key feature of the above example is that, although $(\partial_M\langle y\rangle,G)=\emptyset$, we have $(\partial_M\langle y\rangle,V)\neq\emptyset$. This leads us to suspect that this kind of issue should not present itself if all edge groups have trivial relative Morse boundary in the incident \emph{vertex} groups.

We prove this guess under the following, potentially stronger assumption.

\begin{defn*}[Relatively wide]
Let $H\leq G$ be finitely generated groups, with $H$ undistorted in $G$.
\begin{itemize}
\item The group $G$ is \emph{wide} if none of its asymptotic cones $G_{\om}$ have cut points \cite{Drutu-Sapir-Top}.
\item We say that $H$ is \emph{relatively wide} in $G$ if, for every asymptotic cone $G_{\om}$, no two points of the limit $H_{\om}\cu G_{\om}$ are separated by a cut point of $G_{\om}$.
\end{itemize}
Note that $H$ is relatively wide in $G$ as soon as either $H$ or $G$ is wide.
\end{defn*}

Examples of wide groups include one-ended groups satisfying a law (e.g.\ solvable, uniformly amenable, Burnside, etc) \cite{Drutu-Sapir-Top}, one-ended groups with infinite centre, and various higher-rank lattices \cite{Drutu-Mozes-Sapir}.

Wide groups have empty Morse boundary. It is a well-known open question whether the converse holds. We record here the relative version of this question, as it might be easier to find counter-examples in the relative case.

\begin{quest*}
Let $H\leq G$ be finitely generated groups, with $H$ undistorted in $G$ and $(\partial_MH,G)=\emptyset$. Is $H$ relatively wide in $G$?
\end{quest*}

Conversely, it is easy to see that $(\partial_MH,G)=\emptyset$ holds as soon as $H$ is relatively wide in $G$. We can now state our second main result.

\begin{thmintro}\label{embedding intro}
Let a finitely generated group $G$ split as a graph of groups. Consider a vertex group $V\leq G$. Suppose that all incident edge groups $E\leq V$ are finitely generated, undistorted in $G$, and relatively wide in $V$. Then:
\begin{enumerate}
\item $V$ is undistorted in $G$ and $(\partial_MV,G)=\partial_MV$; 
\item the inclusion $\partial_MV\hookrightarrow\partial_MG$ is a topological embedding.
\end{enumerate}
\end{thmintro}

\begin{rmk*}
We emphasise that \emph{relative wideness} of $E$ in $V$ can be rephrased purely in terms of divergence: it is equivalent to the statement that geodesics in (a Cayley graph of) $E$, viewed as uniform quasi-geodesics in $V$, have \emph{uniform linear divergence} (see Lemma~\ref{rel as cut pts lem}(3)).
\end{rmk*}

\begin{corintro}\label{cor intro 2}
If every vertex group satisfies the assumptions of Theorem~\ref{embedding intro}, then every connected component of $\partial_MG$ is either a singleton or homeomorphic to a connected component of the Morse boundary of a vertex group.
\end{corintro}

Corollary~\ref{cor intro} gives many examples of groups with totally disconnected Morse boundary, but it does not further describe the topological spaces that may arise as boundaries. The remarkable \cite[Theorem~1.4]{Charney-Cordes-Sisto} shows instead that, for any finitely generated group $G$, the boundary $\partial_MG$ is an $\om$--Cantor space as soon as it is totally disconnected, non-compact, $\s$--compact, and contains a Cantor subspace. The last property is generally not hard to obtain: for instance, using acylindrical hyperbolicity of $G$ (if given) to construct a stable free subgroup of $G$ \cite{Sisto-Z,DGO}, or by applying Theorem~\ref{embedding intro} to any Cantor subspace that boundaries of vertex groups may have.

This suggests studying the following problem.

\begin{quest*}
In the setting of Corollary~\ref{cor intro 2}, suppose that all vertex groups have $\s$--compact Morse boundary. Is $\partial_MG$ then $\s$--compact?
\end{quest*}

\medskip
{\bf Acknowledgements.} The authors are grateful to the referee for their helpful suggestions. Fioravanti thanks the Max Planck Institute for Mathematics in Bonn, the University of Bonn and Ursula Hamenst\"adt for their hospitality and financial support while this work was being completed. Karrer was supported by the Israel Science Foundation (grant no.\ 1562/19).

\section{Preliminaries.}

To economise on constants, we will speak of $C$--quasi-geodesics when referring to $(C,C)$--quasi-geodesics. In the whole section, $X$ and $Y$ are proper geodesic metric spaces.

We denote closed metric balls by $\mc{B}(x,r)$ and closed metric neighbourhoods of subsets by $\mc{N}(A,r)$. Where necessary, we may add a subscript $\mc{B}_X(x,r)$ or $\mc{N}_Y(A,r)$ to specify the relevant space.

\subsection{Morse boundaries.}

We refer the reader to \cite{Cordes} for further details.

A quasi-geodesic $\g\cu X$ is \emph{N--Morse} for a function $N\colon[1,+\infty)\ra[0,+\infty)$ if every $C$--quasi-geodesic with endpoints on $\g$ is contained in the $N(C)$--neighbourhood of $\g$. The function $N$ is usually referred to as a \emph{Morse gauge} for $\g$.

Fix a basepoint $p\in X$. We define $\partial_M^NX_p$ as the set of $N$--Morse geodesic rays based at $p$, identifying rays at finite Hausdorff distance. Endowed with the compact-open topology, this space is compact and metrisable. Let $\partial_MX_p$ be the union of all spaces $\partial_M^NX_p$, as $N$ varies among all possible Morse gauges. We define a topology on $\partial_MX_p$ as follows: a subset $U\cu\partial_MX_p$ is open (resp.\ closed) if and only if all intersections $U\cap\partial_M^NX_p$ are open (resp.\ closed). 

If $q\in X$ is a different basepoint, we have a natural homeomorphism $\partial_MX_p\ra\partial_MX_q$ given by pairing rays at finite Hausdorff distance. Thus, the space $\partial_MX_p$ is independent of the choice of $p$ and we simply denote it by $\partial_MX$. We refer to $\partial_MX$ as the \emph{Morse boundary} of $X$. 

We record here the following standard properties of Morse quasi-geodesics for later use.

\begin{lem}\label{nearby ray}
\begin{enumerate}
\item[]
\item Let $\alpha\cu X$ be an $N$--Morse geodesic. Let $\beta\cu X$ be a $C$--quasi-geodesic with endpoints at distance $\leq C$ from those of $\alpha$. Then $d_{\rm Haus}(\alpha,\beta)\leq D$, where $D$ only depends on $C$ and $N$.
\item Let $\alpha\cu X$ be an $N$--Morse geodesic ray. Let $\beta\cu X$ be $C$--quasi-geodesic ray with the same starting point and $d_{\rm Haus}(\alpha,\beta)<+\infty$. Then $d_{\rm Haus}(\alpha,\beta)\leq D$ and $\beta$ is $N'$--Morse, where $D$ and $N'$ only depend on $C$ and $N$.
\item Let $\alpha\cu X$ be an $N$--Morse $C$--quasi-geodesic. Then there exists an $N'$--Morse geodesic $\beta\cu X$ with the same starting point and $d_{\rm Haus}(\alpha,\beta)\leq D$, where $D$ and $N'$ only depend on $C$ and $N$.
\item The restriction of an $N$--Morse $C$--quasi-geodesic to a sub-interval of its domain is always $N'$--Morse, where $N'$ only depends on $N$ and $C$.
\end{enumerate}
\end{lem}
\begin{proof}
Part~(1) is easily deduced from \cite[Lemma~2.1]{Cordes}. Part~(2) is \cite[Corollary~2.5]{Cordes}. Part~(3) can be proved using \cite[Lemma~2.5]{Charney-Sultan} as in the proof of \cite[Lemma~2.9]{Cordes}. Part~(4) is \cite[Lemma~3.1]{Liu}.
\end{proof}

Let $f\colon Y\ra X$ be a quasi-isometric embedding. If $\g$ is a Morse geodesic in $Y$, the quasi-geodesic $f\o\g$ might still not be Morse in $X$. This motivates the following notion of \emph{relative Morse boundary}, which is equivalent to the one from the introduction.

\begin{defn}
Set $(\partial_MY,f):=\{[\g]\in\partial_MY \mid f\o\g \text{ is Morse in } X\}$. We endow $(\partial_MY,f)$ with the subspace topology coming from $\partial_MY$.
\end{defn}

If $[\g]\in(\partial_MY,f)$, then the quasi-geodesic ray $f\o\g$ is Morse in $X$, hence at finite Hausdorff distance from a Morse geodesic ray by Lemma~\ref{nearby ray}(3). This defines an injection:
\[f_*\colon (\partial_MY,f)\hookrightarrow\partial_MX.\]
The example in the introduction shows that $f_*$ is not continuous in general. However, it is always a closed map (equivalently, an open map, since $f_*$ is injective).

\begin{lem}\label{open map}
If $f\colon Y\ra X$ is a quasi-isometric embedding, $f_*\colon (\partial_MY,f)\hookrightarrow\partial_MX$ is a closed map.
\end{lem}
\begin{proof}
Fix a basepoint $q\in Y$ and set $p:=f(q)$. Let $A\cu (\partial_MY,f)$ be a closed subset. We need to show that the intersection $f_*(A)\cap\partial_M^NX_p$ is closed for every Morse gauge $N$. In fact, since $\partial_M^NX_p$ is metrisable, it suffices to show that $f_*(A)\cap\partial_M^NX_p$ is sequentially closed.

Let $\alpha_n\cu Y$ be geodesic rays based at $q$, so that $[\alpha_n]\in A$ and $f_*[\alpha_n]\in\partial_M^NX_p$ for some Morse gauge $N$. Suppose that $f_*[\alpha_n]\ra\xi$ in $\partial_MX$. We need to show that $\xi\in f_*(A)$.

Let $\beta_n\cu X$ be $N$--Morse geodesic rays based at $p$ representing $f_*[\alpha_n]$. By Lemma~\ref{nearby ray}(2), the quasi-geodesic rays $f\o\alpha_n$ are at uniformly finite Hausdorff distance from $\beta_n$, say $\leq D$, and they are uniformly Morse in $X$. Since $f$ is a quasi-isometric embedding, it follows that the $\alpha_n$ are uniformly Morse in $Y$, say $N'$--Morse. 

Now, by the Arzel\`a--Ascoli theorem, we can pass to a subsequence and assume that the $\beta_n$ converge uniformly on compact sets to an $N$--Morse geodesic ray $\beta\cu X$ based at $p$. Similarly, the $\alpha_n$ converge to an $N'$--Morse ray $\alpha\cu Y$ based at $q$. In particular, $[\alpha_n]\ra[\alpha]$ in the topology of $\partial_MY$, hence $[\alpha]\in A$. Since $d_{\rm Haus}(f\o\alpha_n,\beta_n)\leq D$ for every $n$ and $f$ is a quasi-isometric embedding, we have $d_{\rm Haus}(f\o\alpha,\beta)<+\infty$. Hence $\xi=[\beta]=f_*[\alpha]\in f_*(A)$, as required.
\end{proof}

\subsection{Divergence}

It was shown in \cite{ACGH} that Morse quasi-geodesics can equivalently be characterised as quasi-geodesics with completely super-linear divergence. The proof of Theorem~\ref{embedding intro} will require this equivalence to be \emph{effective}, in the sense that the Morse gauge and the divergence function of the quasi-geodesic should determine each other. 

This kind of statement is proved in detail in Cashen's Habilitation thesis \cite[Survey: Corollary~3.5]{Cashen-Hab}, but we also explain here how to deduce it from the proofs of various results in \cite{ACGH}. 

We begin with the definition of divergence.

\begin{defn}
Consider an $L$--quasi-geodesic ray $\g\cu X$ and a parameter $0<\eps<\tfrac{1}{2L}$. The \emph{divergence function} of $\g$ is:
\[\delta_{\g}(r,\eps):=\inf_{s\geq r}\inf\{\text{lengths of paths connecting $\g(s\pm r)$ in $X\setminus \mc{B}(\g(s),\eps r)$}\} \in [0,+\infty].\]
\end{defn}

The proof of \cite[Proposition~5.10]{ACGH} shows the following.

\begin{lem}\label{Morse -> divergence lem}
Let $\g\cu X$ be an $L$--quasi-geodesic ray. Suppose that $\delta_{\g}(r,\eps)\leq Cr$ for some $r,C$ and $\eps<\tfrac{1}{4L}$. Then there exists $s\geq r$ such that $\g(s\pm r)$ are joined by an $L'$--quasi-geodesic avoiding $\mc{B}(\g(s),\eps' r)$, where the constants $L'$ and $\eps'$ depend only on $L,C,\eps$ (and not on $r$).
\end{lem}

\begin{cor}\label{Morse -> divergence cor}
Given constants $L$, $\eps<\tfrac{1}{4L}$, and a Morse gauge $N$, there exists a weakly increasing, diverging function $f$ such that the following holds. For every $N$--Morse $L$--quasi-geodesic ray $\g\cu X$, we have $\delta_{\g}(r,\eps)\geq rf(r)$ for all $r\geq 0$.
\end{cor}
\begin{proof}
Fix $\overline r\geq 0$. Define $g(\overline r)$ as the infimum of the ratio $\delta_{\g}(\overline r,\eps)/\overline r$ as $\g$ varies among all $N$--Morse $L$--quasi-geodesic rays in $X$. Let $\overline\g$ be one such quasi-geodesic satisfying $\delta_{\overline\g}(\overline r,\eps)/\overline r \leq 2g(\overline r)$. 

Lemma~\ref{Morse -> divergence lem} gives $s\geq\overline r$ and an $L'$--quasi-geodesic joining $\overline\g(s\pm \overline r)$ avoiding the ball $\mc{B}(\overline \g(s),\eps' \overline r)$, where $\eps'$ and $L'$ depend only on $L$, $\eps$ and the value $g(\overline r)$. Since $\overline\g$ is $N$--Morse, we must have $\eps' \overline r\leq N(L')$. This implies that $g(r)$ diverges as $r\ra +\infty$.

Now, let $f$ be the largest weakly increasing function with $f\leq g$, namely:
\[ f(x)=\inf_{t\geq x} g(t).\]
Since $g$ diverges, so does $f$. Finally, if $\g\cu X$ is an $N$--Morse $L$--quasi-geodesic ray, it is clear that we have $\delta_{\g}(r,\eps)\geq rg(r)\geq rf(r)$ for all $r\geq 0$.
\end{proof}

The above corollary shows that the divergence function of a Morse quasi-geodesic ray can be bounded uniformly (from below) in terms of the Morse gauge. In order to reverse this kind of result, we need to speak of \emph{contracting} geodesics.

\begin{defn}
Let $\g\cu X$ be a quasi-geodesic with closed image. 
\begin{enumerate}
\item The \emph{nearest-point projection} $\pi_{\g}\colon X\ra 2^{\g}$ is defined by $\pi_{\g}(x)=\{p\in\g \mid d(x,p)=d(x,\g)\}$. Since $X$ is proper and $\g$ is closed, the subset $\pi_{\g}(x)\cu\g$ is always non-empty.
\item If $\rho$ is a sublinear, weakly increasing, non-negative function, we say that $\g$ is \emph{$\rho$--contracting} if, for all $x,y\in X$ with $d(x,y)\leq d(x,\g)$, we have $\diam(\pi_{\g}(x)\cup\pi_{\g}(y))\leq \rho(d(x,\g))$.
\end{enumerate}
\end{defn}

From the proof of \cite[Proposition~5.5]{ACGH}, we obtain:

\begin{lem}\label{divergence -> contracting lem}
Let $\g\cu X$ be an $L$--quasi-geodesic ray with closed image. Consider points $x,y\in X$ with $d(x,y)\leq d(x,\g)$ and projections $x'\in\pi_{\g}(x)$, $y'\in\pi_{\g}(y)$. Then, if $d(x',y')\geq 5L^3$, we have:
\[ 4d(x,\g)\geq \delta_{\g}\big(r,\tfrac{1}{16L}\big), \]
for some $r\in [\tfrac{1}{4L}\cdot d(x',y'),L\cdot d(x',y')]$.
\end{lem}

Finally, the following is \cite[Proposition~4.1]{ACGH}.

\begin{prop}\label{contracting -> Morse}
Let $\g\cu X$ be a quasi-geodesic with closed image. If $\g$ is $\rho$--contracting, then it is $N$--Morse with $N$ depending only on $\rho$.
\end{prop}

We will use the combination of the previous results in the following form.

\begin{lem}\label{Morse vs divergence}
Let $\alpha\cu X$ and $\beta\cu Y$ be $L$--quasi-geodesic rays. Suppose that there exist constants $K\geq 0$ and $0<\eps_1,\eps_2\leq\tfrac{1}{16L}$ such that, for all $r>K$, we have:
\[\delta_{\alpha}(r,\eps_1)\leq K\cdot\delta_{\beta}(r,\eps_2).\]  
If $\alpha$ is $N$--Morse in $X$, then $\beta$ is $N'$--Morse in $Y$, with $N'$ depending only on $N,L,K,\eps_1$.
\end{lem}
\begin{proof}
We will show that $\beta$ is $\rho$--contracting, with $\rho$ depending only on $N,L,K,\eps_1$. The fact that $\beta$ is $N'$--Morse then follows from Proposition~\ref{contracting -> Morse}.

By Corollary~\ref{Morse -> divergence cor}, there exists a weakly increasing, diverging function $f$ such that $\delta_{\alpha}(r,\eps_1)\geq rf(r)$ for all $r\geq 0$. Recall that $f$ depends only on $L,\eps_1,N$. Setting $\eta(r):=rf(r)/K$, we have:
\[\delta_{\beta}(r,\tfrac{1}{16L})\geq\delta_{\beta}(r,\eps_2)\geq \tfrac{1}{K}\cdot\delta_{\alpha}(r,\eps_1) \geq\eta(r),\] 
for all $r>K$. Note that $\eta$ is strictly increasing and $\eta(r)/r\ra+\infty$. Define:
\begin{align*}
\rho'(x)&:=\sup\{ r \mid \eta(r)\leq x\}, & \rho(x)&:=\sup_{t\leq x}\rho'(t).
\end{align*}
Observe that $\rho'(x)/x\ra 0$ for $x\ra+\infty$. In addition, $\rho$ is weakly increasing and $\rho\geq\rho'$. Observing that $\rho'$ is bounded on bounded sets, the fact that $\rho'$ is sublinear implies that $\rho$ is sublinear as well.

Now, consider points $x,y\in Y$ with $d(x,y)\leq d(x,\beta)$ and projections $x'\in\pi_{\beta}(x)$, $y'\in\pi_{\beta}(y)$. Suppose that $d(x',y')> \max\{5L^3,4LK\}$. Recalling that $\eta$ is monotone, Lemma~\ref{divergence -> contracting lem} shows that:
\[4d(x,\beta)\geq\eta(\tfrac{1}{4L}\cdot d(x',y')).\]
It follows that:
\[d(x',y')\leq 4L\cdot\rho'(4d(x,\beta))\leq 4L\cdot\rho(4d(x,\beta)).\]
This shows that $\beta$ is $\rho$--contracting, where $\rho$ depends only on $\eta$, hence only on $L,\eps_1,N,K$.
\end{proof}

\subsection{Relatively wide subgroups.}

Relatively wide subgroups were defined in the introduction. This property will be required in the proof of Theorem~\ref{embedding intro} in the form of part~(3) of the next lemma.

\begin{lem}\label{rel as cut pts lem}
Let $H\leq G$ be finitely generated groups, with $H$ undistorted in $G$. Let $\L$ and $\G$ be Cayley graphs for $H$ and $G$, respectively. Let $i\colon\L\ra\G$ be a quasi-isometric embedding corresponding to the inclusion $H\hookrightarrow G$. Then the following properties are equivalent.
\begin{enumerate}
\item The subgroup $H$ is relatively wide in $G$.

\smallskip
\item For all $C\geq 1$, there exists $K=K(C)\geq 1$ such that the following holds. Let $\alpha\colon [-r,r]\ra\G$ be a $C$--quasi-geodesic with $\alpha(\pm r)\in H\cu\G^{(0)}$. If $r>K$, then $\alpha(\pm r)$ are joined by an edge path $\g\cu\G$ that is disjoint from the ball $\mc{B}_{\G}(\alpha(0),\tfrac{r}{K})$ and has length $|\g|\leq K\cdot r$.

\smallskip
\item There exists $K_0\geq 1$ such that the following holds. Let $\beta\colon[-r,r]\ra\L$ be a geodesic. If $r>K_0$, then $i\o\beta(\pm r)\in\G^{(0)}$ are joined by an edge path $\g\cu\G$ that is disjoint from the ball $\mc{B}_{\G}(i\o\beta(0),\tfrac{r}{K_0})$ and has length $|\g|\leq K_0\cdot r$.
\end{enumerate}
\end{lem}
\begin{proof}
Since $H$ is undistorted, the implication (2)$\Ra$(3) is clear. We show that (3)$\Ra$(1) and (1)$\Ra$(2).

\smallskip
{\bf (1)$\Ra$(2).} Suppose for the sake of contradiction that (2) fails for some constant $C$. Then there exists a sequence of $C$--quasi-geodesics $\alpha\colon [-r_n,r_n]\ra\G$ with $r_n\ra+\infty$, such that every path joining $\alpha_n(\pm r_n)\in H$ in $\G\setminus \mc{B}_{\G}(\alpha_n(0),\tfrac{r_n}{n})$ has length $>nr_n$.

Fix a non-principal ultrafilter $\om$. Let $\G_{\om}$ be the asymptotic cone of $\G$ given by basepoints $\alpha_n(0)$ and scaling factors $\tfrac{1}{r_n}$.  Let $x_-,x_0,x_+\in \G_{\om}$ be the points determined by the sequences $\alpha_n(-r_n),\alpha_n(0),\alpha_n(r_n)$, respectively. The $\alpha_n$ converge to a $C$--bi-Lipschitz curve $\alpha\colon[-1,1]\ra \G_{\om}$ with $\alpha(-1)=x_-$, $\alpha(0)=x_0$ and $\alpha(1)=x_+$.

Since $H$ is relatively wide in $G$ and $x_{\pm}$ are limits of sequences in $H\cu\G$, the points $x_{\pm}$ lie in the same connected component of $\G_{\om}\setminus\{x_0\}$. Since $\G_{\om}$ is a geodesic space, it is locally path connected, hence there exists a continuous path in $\G_{\om}\setminus\{x_0\}$ joining $x_{\pm}$. Being compact, this path misses the $4\eps$--ball around $x_0$ for some $\eps>0$. Thus, we can discretise it to a sequence $x_1,\dots,x_k\in \G_{\om}\setminus\mc{B}_{\G_{\om}}(x_0,4\eps)$ with $d(x_i,x_{i+1})\leq\eps$ and $x_1=x_-$, $x_k=x_+$.

Choose approximations $x_i(n)\in\G$ with $d(x_i(n),x_{i+1}(n))\leq 2\eps r_n$ and $d(x_i(n),\alpha_n(0))\geq 3\eps r_n$, so that $x_1(n)=\alpha_n(-r_n)$ and $x_k(n)=\alpha_n(r_n)$. Joining consecutive points by geodesics, we obtain a path of length $\leq 2k\eps r_n$ avoiding $\mc{B}_{\G}(\alpha_n(0),\eps r_n)$ and connecting $\alpha_n(\pm r_n)$. For large $n$, we have $\tfrac{1}{n}<\eps$ and $n>2k\eps$, which contradicts our initial assumptions.

\smallskip
{\bf (3)$\Ra$(1).} Suppose towards a contradiction that $H$ is not relatively wide in $G$. Thus, there exist an asymptotic cone $\G_{\om}$ and two points $x,y\in H_{\om}\cu\G_{\om}$ that are separated by a third point $z\in\G_{\om}$. 

Write $x=(x_n)$ and $y=(y_n)$ with $x_n,y_n\in H\cu\G^{(0)}$ and choose geodesics $\beta_n\colon I_n\ra\L$ joining $x_n$ and $y_n$, for intervals $I_n$. The paths $i\o\beta_n$ are uniform quasi-geodesics in $\G$, so they converge to a bi-Lipschitz path $\alpha\colon I\ra H_{\om}\cu\G_{\om}$ from $x$ to $y$. Since $z$ separates $x$ and $y$, it must lie in $\alpha(I)$. 

Thus, possibly replacing $x$ and $y$ with other points of $\alpha(I)$, reparametrising $\alpha$ and $\beta_n$, and modifying the scaling factors in the definition of $\G_{\om}$ (though not their growth rate), we can assume that $I=[-1,1]$, $\alpha(0)=z$ and $I_n=[-r_n,r_n]$ for some sequence $r_n\ra+\infty$. In particular, the points $i\o\beta_n(0)$ converge to $z$.

Since (3) holds, there exist edge paths $\g_n\cu\G$ joining $x_n$ to $y_n$, avoiding $\mc{B}_{\G}(i\o\beta_n(0),\tfrac{r_n}{K_0})$ and having length $|\g_n|\leq K_0\cdot r_n$. These paths can be discretised to sequences of points $w_1(n),\dots,w_k(n)\in\G$, with $k$ independent of $n$, such that $d(w_i(n),i\o\beta_n(0))>\tfrac{r_n}{K_0}$ and $d(w_i(n),w_{i+1}(n))\leq \tfrac{r_n}{2K_0}$ for all $i$. We choose these points so that $w_1(n)=x_n$ and $w_k(n)=y_n$.

Let $w_i\in\G_{\om}$ be the limit of the sequence $w_i(n)$. Note that $d(w_i,z)\geq\tfrac{1}{K_0}$ and $d(w_i,w_{i+1})\leq\tfrac{1}{2K_0}$, so every geodesic joining $w_i$ and $w_{i+1}$ in $\G_{\om}$ avoids $\mc{B}_{\G_{\om}}(z,\tfrac{1}{2K_0})$. Concatenating such geodesics, we obtain a path in $\G_{\om}$ from $w_1=x$ to $w_k=y$ avoiding $\mc{B}_{\G_{\om}}(z,\tfrac{1}{2K_0})$, which contradicts the assumption that $x$ and $y$ are separated by $z$.
\end{proof}

\subsection{Graphs of groups.}

Throughout the paper, we are interested in finitely generated groups $G$ that split as a graph of groups. This is equivalent to the fact that $G$ admits a non-elliptic action on a simplicial tree $G\acts T$ without inversions \cite{Serre}. In this case, there is a unique smallest $G$--invariant subtree \cite{Culler-Morgan}. Restricting to it, we can further assume that the action $G\acts T$ is \emph{minimal}, i.e.\ that no proper subtree is $G$--invariant.

In this subsection, we consider the following setting.

\begin{ass}\label{ass1}
Let $G$ be a group generated by a finite subset $S_G\cu G$ with $1\in S_G$ and $S_G=S_G^{-1}$. Let $\G$ be the corresponding Cayley graph of $G$, endowed with its graph metric $d_{\G}$. Suppose that we have a non-elliptic, minimal action without inversions on a simplicial tree $G\acts T$.
\end{ass}

Vertex stabilisers for the action $G\acts T$ will be referred to as \emph{vertex groups}, and usually denoted by $V$. Similarly, \emph{edge groups} are $G$--stabilisers of edges of $T$, and will usually be denoted by $E$ when the corresponding edge is not specified. To avoid confusion between edges of $\G$ and $T$, we will denote the latter by fraktur letters $\mf{e}$.

\begin{rmk}\label{finite graph of groups}
The action $G\acts T$ is cocompact. Indeed, fixing a basepoint $p\in T$, let $K$ be the convex hull of $S_G\cdot p$. This is a compact subtree of $T$. Since $S_G$ generates $G$, the subset $G\cdot K$ is connected, hence a $G$--invariant subtree. By minimality, we obtain $T=G\cdot K$.
\end{rmk}

The following results are certainly well-known to experts, but we were not able to locate proofs in the literature at the required level of generality. Our arguments are inspired by the proof of \cite[Proposition~1.2]{Bow-JSJ} in the case of hyperbolic groups with quasi-convex edge groups.

\begin{lem}\label{path decomposition}
Consider a vertex $p\in T$ and let $V$ be its stabiliser. There exist a constant $D_p$ and stabilisers $E_1,\dots,E_k$ of edges incident to $p$ such that the following holds. Every path $\g\cu\G$ joining points of $V$ can be decomposed as a concatenation $\g_1\g_2\dots\g_m$ with the following property. The endpoints of each $\g_i$ lie in the $D_p$--neighbourhood in $\G$ of a coset $v_iE_{j_i}$ with $v_i\in V$ and $1\leq j_i\leq k$.
\end{lem}
\begin{proof}
By Remark~\ref{finite graph of groups}, there are finitely many $V$--orbits of edges of $T$ incident to $p$. Let $\mf{e}_1,\dots,\mf{e}_k$ be a finite list of representatives and let $E_1,\dots,E_k$ be their stabilisers. Let $K\cu T$ be the convex hull of $S_G\cdot p$. If $\mf{e}\cu T$ is an edge, define $\Om(\mf{e}):=\{g\in G\mid \mf{e}\cu gK\}$. 

\smallskip
{\bf Claim:} \emph{there exists a constant $D_p$ such that, for every edge $\mf{e}\cu T$ incident to $p$, the set $\Om(\mf{e})$ is contained in the $D_p$--neighbourhood in $\G$ of some coset $vE_i$ with $v\in V$ and $1\leq i\leq k$.}

\smallskip\noindent
\emph{Proof of claim.}
Observe that $\Om(\mf{e})$ is a union of at most $N$ right cosets of $E$, where $E$ is the stabiliser of $\mf{e}$ and $N$ is the number of edges in the compact tree $K$. Thus, there exists a constant $D_p$ such that, for each $1\leq i\leq k$, the set $\Om(\mf{e}_i)$ is contained in the $D_p$--neighbourhood of $E_i$ in $\G$. 

For an arbitrary edge $\mf{e}\ni p$, we can write $\mf{e}=v\mf{e}_i$ for some $v\in V$ and some $i$. In this case, $\Om(\mf{e})=\Om(v\mf{e}_i)=v\Om(\mf{e}_i)$ is contained in the $D_p$--neighbourhood of the left coset $vE_i$.
\hfill$\blacksquare$

\smallskip
Now, we define a continuous $G$--equivariant map $f_p\colon\G\ra T$ as follows. For every $s\in S_G$, let $\pi_s\cu K$ be the geodesic from $p$ to $sp$. If $g\in G$, we set $f_p(g)=gp$. Then, on each edge $[g,gs]\cu\G$ with $g\in G$ and $s\in S_G$, we define $f_p$ as the linear parametrisation of the path $g\pi_s$.

Up to decomposing $\g$, we can assume that $\g$ meets $V$ only at its endpoints. We can also suppose that $\g$ is not a single edge, or the statement is clear. Thus, $f_p\o\g$ is a nontrivial path in the tree $T$ beginning and ending at the basepoint $p$. 

Let $e_1,\dots,e_m\cu\G$ be all edges of $\g$ for which $f_p(e_i)$ contains $p$, in the order in which they appear along $\g$. Thus, $e_1$ and $e_m$ are necessarily the initial and terminal edge of $\g$, respectively. Note that the geodesic $f_p(e_i)\cu T$ meets $p$ exactly once and not at its endpoints, except for $e_1$ and $e_m$. 

Define $\g_i\cu\g$ as the sub-segment between the initial vertex of $e_i$ and the initial vertex of $e_{i+1}$. The portion of the path $f_p(\g_i\cup e_{i+1})$ between the two occurrences of $p$
is contained in a connected component of $T\setminus\{p\}$. Thus, it must begin and end by crossing the same edge of $T$ incident to $p$. 

We name this edge $\mf{f}_i$. Say the endpoints of $e_i\cu\G$ are $g_i\in G$ and $g_is_i$ with $s_i\in S_G$. Since $\mf{f}_i\cu f_p(e_i)\cap f_p(e_{i+1})$, we have $\mf{f}_i\cu g_i\pi_{s_i}\cap g_{i+1}\pi_{s_{i+1}}$, hence $\{g_i,g_{i+1}\}\cu\Om(\mf{f}_i)$. Since $g_i$ and $g_{i+1}$ are the endpoints of $\g_i$, the Claim concludes the proof.
\end{proof}

If $\alpha\cu\G$ is a path, we denote by $|\alpha|$ its \emph{length}, i.e.\ the number of edges that it contains. Recall that a finitely generated subgroup of $G$ is \emph{undistorted} if the inclusion in $G$ is a quasi-isometric embedding.

\begin{lem}\label{undistorted vertex groups lem}
Let $p\in T$ be a vertex with stabiliser $V$. Suppose that the stabiliser of every edge $p\in\mf{e}\cu T$ is finitely generated and undistorted in $G$. Then $V$ is finitely generated and undistorted.
\end{lem}
\begin{proof}
We will prove the following.

\smallskip
{\bf Claim:} \emph{there exists a constant $L$ such that, for every $v\in V$, there exists a path $\alpha\cu\G$ from the identity to $v$ such that $\alpha$ is contained in the $L$--neighbourhood of $V$ and $|\alpha|\leq L\cdot d_{\G}(1,v)$.}

\smallskip
Assuming the claim, we define $\Sigma$ as the intersection of $V$ and the $(2L+1)$--ball in $\G$ centred at the identity. Since $\G$ is locally finite, the set $\Sigma$ is finite. Denoting by $d_{\Sigma}$ the word metric on $V$ induced by $\Sigma$, the claim shows that $d_{\Sigma}(1,v)\leq L\cdot d_{\G}(1,v)$ for every $v\in V$. Thus, $\Sigma$ generates $V$ and $V$ is undistorted in $G$.

Now, let us prove the claim. Let the constant $D_p$ and the edge groups $E_1,\dots,E_k\leq V$ be those provided by Lemma~\ref{path decomposition}. Let $\g\cu\G$ be a geodesic from the identity to some element $v\in V$. Decompose $\g=\g_1\g_2\dots\g_m$ as in Lemma~\ref{path decomposition}, with the endpoints of $\g_i$ in the $D_p$--neighbourhood of a coset $v_iE_{j_i}$.

Since $E_1,\dots,E_k$ are finitely generated and undistorted, there exists a constant $L$ such that, for all $i$ and all $x,y\in E_i$, the points $x$ and $y$ are joined by a path $\beta\cu\G$ contained in the $L$--neighbourhood of $E_i$ and satisfying $|\beta|\leq L\cdot d_{\G}(x,y)$. The same holds if $x$ and $y$ lie in a left coset of some $E_i$. 

It follows that we can replace each $\g_i$ with a path $\g_i'$ with the same endpoints so that $\g_i'$ is contained in the $L$--neighbourhood of $v_iE_{j_i}\cu V$ and $|\g_i'|\leq L\cdot |\g_i|$. Define $\alpha$ as the concatenation of the $\g_i'$. Then, it is clear that $\alpha$ is contained in the $L$--neighbourhood of $V$ and $|\alpha|\leq L\cdot |\g|=L\cdot d_{\G}(1,v)$. This proves the claim, hence the lemma.
\end{proof}

\begin{cor}\label{undistorted vertex groups cor}
If all edge groups are finitely generated and undistorted in $G$, then all vertex groups are finitely generated and undistorted in $G$.
\end{cor}

The fact that vertex groups are finitely generated as soon as edge groups are finitely generated is also proved in detail e.g.\ in \cite[Lemma~8.32, p.~218]{Cohen} and \cite{Dicks-Dunwoody}.

\section{Proof of Theorem~\ref{components intro}.}

\begin{ass}\label{ass3}
Let $G$ be a finitely generated group. Let $G\acts T$ be a minimal action on a simplicial tree without inversions. Suppose that all edge-stabilisers $E\leq G$ are finitely generated, undistorted, and satisfy $(\partial_ME,G)=\emptyset$.
\end{ass}

Fix a finite generating set for $G$ and let $\G$ be the corresponding Cayley graph. Choose a basepoint $p\in T^{(0)}$ and let $f_p\colon G\ra T$ denote the orbit map $f_p(g)=gp$.

The complement in $T$ of an open edge has exactly two connected components, which we refer to as \emph{halfspaces}. Let $\E(T)$ denote the set of (unoriented) edges of $T$, and let $\H(T)$ be the set of halfspaces. Note that $\H(T)$ is naturally in bijection with the set of \emph{oriented} edges of $T$. In particular, every edge $\mf{e}\in\E(T)$ gives rise to two halfspaces $\mf{h},\mf{h}^*$. The complement of $\mf{h}$ is always denoted $\mf{h}^*$.

We denote by $G_{\mf{h}}\leq G$ the stabiliser of the halfspace $\mf{h}$. Since $G$ acts without inversions, $G_{\mf{h}}$ coincides with the stabiliser $G_{\mf{e}}$ of the edge $\mf{e}$ associated to $\mf{h}$. 

Note that every $\mf{h}\in\H(T)$ gives rise to a $G_{\mf{h}}$--invariant partition $G=f_p^{-1}(\mf{h})\sqcup f_p^{-1}(\mf{h}^*)$.

\begin{lem}\label{AIS lem}
We can choose a subgraph $\G(\mf{e})\cu\G$ for every $\mf{e}\in\E(T)$ so that the following hold:
\begin{enumerate}
\item $\G(\mf{e})$ is connected, $G_{\mf{e}}$--invariant and $G_{\mf{e}}$--cocompact;
\item $\G(g\mf{e})=g\G(\mf{e})$ for all $g\in G$ and $\mf{e}\in\E(T)$;
\item if $\mf{h},\mf{h}^*$ are the two halfspaces determined by $\mf{e}$, then $\G(\mf{e})$ contains every edge of $\G$ with an endpoint in $f_p^{-1}(\mf{h})$ and the other in $f_p^{-1}(\mf{h}^*)$;
\item for every $r\geq 0$, we have $\mc{N}_{\G}(f_p^{-1}(\mf{h}), r)\cap\mc{N}_{\G}(f_p^{-1}(\mf{h}^*), r)\cu\mc{N}_{\G}(\G(\mf{e}),r)$;
\item each edge $e\cu\G$ lies in the subgraph $\G(\mf{e})$ for only finitely many edges $\mf{e}\in\E(T)$.
\end{enumerate}
\end{lem}
\begin{proof}
Recall that edge-stabilisers are finitely generated. If $S_{\mf{e}}$ is a finite generating set of $G_{\mf{e}}$, we can construct a subgraph $\G(\mf{e})\cu\G$ satisfying (1) by choosing paths joining the identity $1\in G$ to the elements of $S_{\mf{e}}$, and taking the union of all their $G_{\mf{e}}$--translates. Defining $\G(\mf{e})$ in this fashion for one edge $\mf{e}$ in every $G$--orbit of edges of $T$, and then setting $\G(g\mf{e}):=g\G(\mf{e})$, guarantees that condition~(2) is also satisfied.

Now, let us ensure that (3) holds. For this, it suffices to show that there are only finitely many $G_{\mf{e}}$--orbits of edges of $\G$ with an endpoint in $f_p^{-1}(\mf{h})$ and the other in $f_p^{-1}(\mf{h}^*)$, as these can then be added to $\G(\mf{e})$. If $S_G$ is the finite generating set of $G$ giving rise to $\G$, then every such edge is of the form $[g,gs]$ with $gp\in\mf{h}$, $gsp\in\mf{h}^*$ and $s\in S_G$. Note that $g^{-1}\mf{e}$ is then one of the finitely many edges separating $p$ and $sp$ for some $s\in S_G$. It follows that $g$ lies in a finite union of right cosets of $G_{\mf{e}}$, as required.

Note that (4) follows from (3): if a point $x\in\G$ lies in the $r$--neighbourhood of both $f_p^{-1}(\mf{h})$ and $f_p^{-1}(\mf{h}^*)$, then it lies on a path of length $\leq r$ joining $f_p^{-1}(\mf{h})$ to $f_p^{-1}(\mf{h}^*)$. Since $G=f_p^{-1}(\mf{h})\sqcup f_p^{-1}(\mf{h}^*)$, condition~(3) implies that such a path must contain an edge of $\G(\mf{e})$. Hence $x$ lies in the $r$--neighbourhood of $\G(\mf{e})$.

Finally, (5) follows from (1). Fix some $\mf{e}\in\E(T)$. Suppose that $e\cu g\G(\mf{e})$ for some $g\in G$. Since $g^{-1}e\cu\G(\mf{e})$ and $G_{\mf{e}}$ acts cocompactly on $\G(\mf{e})$, we deduce that $g$ must lie in a finite union of left cosets of $G_{\mf{e}}$. In other words, $e$ is contained in $\G(\mf{e}')$ for only finitely many edges $\mf{e}'$ in the $G$--orbit of $\mf{e}$. Since there are only finitely many $G$--orbits in $\E(T)$, by Remark~\ref{finite graph of groups}, this proves (5).
\end{proof}

Let us choose $\G(\mf{e})\cu\G$ for every $\mf{e}\in\E(T)$ as in Lemma~\ref{AIS lem}. If $\mf{h}$ is one of the two halfspaces determined by $\mf{e}$, it is convenient to define $\G(\mf{h})\cu\G$ as the set $f_p^{-1}(\mf{h})\setminus\G(\mf{e})\cu G$ along with all (half-open) edges of $\G\setminus\G(\mf{e})$ that it intersects. Thus, we have a $G_{\mf{e}}$--invariant partition:
\[ \G = \G(\mf{h}) \sqcup \G(\mf{e}) \sqcup \G(\mf{h}^*).\]

\begin{lem}\label{stay close to w}
For every Morse gauge $N$, there exists a constant $D(N)$ with the following property. If $\beta\cu\G$ is an $N$--Morse geodesic with endpoints in $\G(\mf{e})$ for some $\mf{e}\in\E(T)$ (resp.\ in $\G(\mf{h})$ for some $\mf{h}\in\H(T)$), then $\beta$ is contained in the $D(N)$--neighbourhood of $\G(\mf{e})$ (resp.\ of $\G(\mf{h})$).
\end{lem}
\begin{proof}
By Lemma~\ref{nearby ray}(4) and Lemma~\ref{AIS lem}(3), the statement about $\G(\mf{h})$ follows from that on $\G(\mf{e})$.

Let $\mf{e}_1,\dots,\mf{e}_k$ be representatives for the orbits of $G\acts\E(T)$ (recall Remark~\ref{finite graph of groups}). Since edge-stabilisers are undistorted, there exists a constant $L$ such that any two points of $\G(\mf{e}_i)$ can be connected by a path entirely contained in $\G(\mf{e}_i)$ that is an $L$--quasi-geodesic in $\G$.

Now, suppose then that $\beta$ has endpoints $x,y\in\G(\mf{e})$. There exist $g\in G$ and $1\leq i\leq k$ such that $\G(\mf{e})=g\G(\mf{e}_i)$. Thus, there exists an $L$--quasi-geodesic $\alpha\cu\G$ connecting $x$ and $y$ within $\G(\mf{e})$. By Lemma~\ref{nearby ray}(1), the Hausdorff distance between $\beta$ and $\alpha$ is at most $D(N)$, where $D(N)$ depends only on $N$ and $L$. Since $\alpha\cu\G(\mf{e})$, this proves the lemma.
\end{proof}

\begin{cor}\label{choose a side}
Let $\g\cu\G$ be a Morse geodesic ray. Then:
\begin{enumerate}
\item for every edge $\mf{e}\in\E(T)$, the intersection $\g\cap\G(\mf{e})$ is compact; 
\item for every halfspace $\mf{h}\in\mc{H}(T)$, a sub-ray of $\g$ is entirely contained in either $\G(\mf{h})$ or $\G(\mf{h}^*)$;
\item if a sub-ray of $\g$ is contained in $\G(\mf{h})$, then $\g$ must get arbitrarily far from $\G(\mf{h}^*)$.
\end{enumerate}
\end{cor}
\begin{proof}
If $\g\cap\G(\mf{e})$ were non-compact, it would be unbounded and Lemma~\ref{stay close to w} would show that $\g$ stays at bounded distance from $\G(\mf{e})$. This would contradict the assumption that $(\partial_MG_{\mf{e}},G)=\emptyset$. This proves part~(1). Parts~(2) and~(3) then follow, respectively, from properties~(3) and~(4) in Lemma~\ref{AIS lem}.
\end{proof}

By Corollary~\ref{choose a side}(3), which of the sets $\G(\mf{h})$ and $\G(\mf{h}^*)$ contains a sub-ray of $\g$ does not change if we replace $\g$ with a ray at finite Hausdorff distance. This leads us to consider the following (well-defined) subset of $\partial_MG$ for each $\mf{h}\in\H(T)$:
\[M(\mf{h})=\{[\g]\in\partial_MG \mid \text{$\g$ is eventually contained in $\G(\mf{h})$}\}.\]

\begin{lem}\label{clopen partition}
For every $\mf{h}\in\H(T)$, we have a partition $\partial_MG=M(\mf{h})\sqcup M(\mf{h}^*)$ into closed subsets.
\end{lem}
\begin{proof}
From the above discussion, it is clear that $M(\mf{h})$ and $M(\mf{h}^*)$ are disjoint and cover $\partial_MG$. We only need to prove that they are closed. Choosing some $x\in\G(\mf{h})$ as basepoint, it suffices to show that $M(\mf{h})\cap\partial_M^N\G_x$ is sequentially closed for every Morse gauge $N$. 

Let $\g_n\cu\G$ be a sequence of $N$--Morse geodesic rays based at $x$ that converge uniformly on compact sets to an $N$--Morse geodesic ray $\g\cu\G$. If $[\g_n]\in M(\mf{h})$ for every $n$, then Lemma~\ref{stay close to w} guarantees that the $\g_n$ are all contained in the $D(N)$--neighbourhood of $\G(\mf{h})$. The same neighbourhood must then contain $\g$ and, by Corollary~\ref{choose a side}(3), we conclude that $[\g]\in M(\mf{h})$.
\end{proof}

If $\xi\in\partial_MG$, we introduce the following subset of $\H(T)$:
\begin{align*}
\s(\xi):=&~\{\mf{h}\in\H(T) \mid \xi\in M(\mf{h})\} \\
=&~\{\mf{h}\in\H(T) \mid \text{if $\xi=[\g]$, then $\g$ is eventually contained in $\G(\mf{h})$}\}.
\end{align*}
Note that $\s(\xi)$ satisfies the following two properties:
\begin{enumerate}
\item if $\mf{e}\in\E(T)$, then exactly one of the two halfspaces determined by $\mf{e}$ lies in $\s(\xi)$;
\item if $\mf{h}_1,\mf{h}_2\in\s(\xi)$, then $\mf{h}_1\cap\mf{h}_2\neq\emptyset$.
\end{enumerate}
A subset of $\H(T)$ with these properties is known as an \emph{ultrafilter} \cite[Definition~2.1]{Sageev-notes}.

For every vertex $x\in T^{(0)}$, an important example of ultrafilter is the set:
\[\s_x:=\{\mf{h}\in\H(T) \mid x\in\mf{h}\}.\]
Conversely, every ultrafilter $\s\cu\H(T)$ that does not contain infinite descending chains of halfspaces is of the form $\s_x$ for a vertex $x\in T^{(0)}$ (see for instance \cite[Proposition~2.1]{Sageev-notes}).

\begin{rmk}\label{same ultrafilter rmk}
If $\xi,\eta\in\partial_MG$ are in the same connected component, then $\s(\xi)=\s(\eta)$. Indeed, Lemma~\ref{clopen partition} implies that, for every $\mf{h}\in\H(T)$, we have either $\{\xi,\eta\}\cu M(\mf{h})$ or $\{\xi,\eta\}\cu M(\mf{h}^*)$.
\end{rmk}

\begin{lem}\label{line ultrafilter}
Let $\alpha\cu\G$ be an $N$--Morse geodesic line whose endpoints at infinity $\alpha^{\pm}\in\partial_MG$ lie in the same connected component of $\partial_MG$. Set $\s:=\s(\alpha^+)$, which coincides with $\s(\alpha^-)$ by Remark~\ref{same ultrafilter rmk}. Then, the following hold:
\begin{enumerate}
\item $\alpha$ is contained in the intersection of the $D(N)$--neighbourhoods of the subsets $\G(\mf{h})\cu\G$ with $\mf{h}\in\s$ (here $D(N)$ is the constant introduced in Lemma~\ref{stay close to w});
\item there exists a vertex $x\in T$ such that $\s=\s_x$;
\item $\alpha$ stays at bounded distance from the stabiliser $G_x\leq G$.
\end{enumerate}
\end{lem}
\begin{proof}
Part~(1) is immediate from Lemma~\ref{stay close to w}. We will prove part~(2) by showing that $\s$ does not contain any infinite descending chains of halfspaces.

Fix a vertex $g\in\alpha\cu\G$. Consider the ultrafilter $\s_{gp}\cu\H(T)$ determined by the vertex $gp\in T$. If $\mf{h}\in\s\setminus\s_{gp}$, then $g\in f_p^{-1}(\mf{h}^*)\cap\alpha$, while $\alpha$ is contained in the $D(N)$--neighbourhood of $\G(\mf{h})$ by part~(1). If $\mf{e}$ is the edge corresponding to $\mf{h}$, Lemma~\ref{AIS lem}(4) guarantees that $\G(\mf{e})$ meets the $D(N)$--ball around $g$. By Lemma~\ref{AIS lem}(5), the latter can occur only for finitely many edges of $T$. 

This shows that the set $\s\setminus\s_{gp}$ is finite. Since $\s_{gp}$ does not contain infinite descending chains, neither does $\s$. This proves part~(2).

Finally, let us prove part~(3). Let $x\in T$ be the vertex provided by part~(2). Let $\tau_x\cu\s_x$ be the subset of halfspaces whose corresponding edge of $T$ is incident to $x$. Let $\Om\cu\G$ be the intersection of the $D(N)$--neighbourhoods of the sets $\G(\mf{h})$ with $\mf{h}\in\tau_x$. Part~(1) guarantees that $\alpha\cu\Om$. Since $G_x$ fixes $x$, it leaves invariant $\tau_x$, and thus it also leaves invariant $\Om\cu\G$. In order to prove part~(3), it suffices to show that the action $G_x\acts\Om$ is cocompact. This implies that the Hausdorff distance between $\Om$ and $G_x$ is finite, so $\alpha$ then stays at bounded distance from $G_x$.

Observe that $\{x\}=\bigcap_{\mf{h}\in\tau_x}\mf{h}$, hence $f_p^{-1}(x)=\bigcap_{\mf{h}\in\tau_x}f_p^{-1}(\mf{h})$. The latter is either empty (if $x$ is not in the same $G$--orbit as our basepoint $p\in T$), or a single $G_x$--orbit. Thus, it suffices to show that $G_x$ acts with finitely many orbits on the difference $\Om^{(0)}\setminus f_p^{-1}(x)$.

For every vertex $y\in\Om^{(0)}\setminus f_p^{-1}(x)$, there exists $\mf{h}\in\tau_x$ such that $y$ lies in the intersection of $f_p^{-1}(\mf{h}^*)$ and the $D(N)$--neighbourhood of $\G(\mf{h})$. If $\mf{e}$ is the edge corresponding to $\mf{h}$, Lemma~\ref{AIS lem}(4) shows that $y$ lies in the $D(N)$--neighbourhood of $\G(\mf{e})$. Recall that there are only finitely many $G_x$--orbits of edges $\mf{e}$ corresponding to elements of $\tau_x$, that these edges satisfy $G_{\mf{e}}\leq G_x$, and that $G_{\mf{e}}$ acts cocompactly on $\G(\mf{e})$. This shows that $G_x$ acts with finitely many orbits on $\Om^{(0)}\setminus f_p^{-1}(x)$, as required.
\end{proof}

\begin{cor}\label{component cor}
If a connected component $\mc{C}\cu\partial_MG$ is not a singleton, then there exists a vertex $x\in T$ such that $\mc{C}$ is contained in the image of the natural inclusion $(\partial_MG_x,G)\hookrightarrow\partial_MG$.
\end{cor}
\begin{proof}
By Remark~\ref{same ultrafilter rmk}, there exists an ultrafilter $\s\cu\H(T)$ such that $\s=\s(\xi)$ for all $\xi\in\mc{C}$. Any two points of $\mc{C}$ are endpoints at infinity of a Morse geodesic line in $\G$, by \cite[Proposition~3.11]{Cordes}. Thus, Lemma~\ref{line ultrafilter}(2) shows that $\s=\s_x$ for some $x\in T^{(0)}$, and Lemma~\ref{line ultrafilter}(3) guarantees that every point of $\mc{C}$ is represented by a ray at bounded distance from $G_x$.
\end{proof}

Along with the following observation, Corollary~\ref{component cor} immediately implies Theorem~\ref{components intro}.

\begin{lem}
If $x,y\in T$ are distinct vertices, then $(\partial_MG_x,G)\cap(\partial_MG_y,G)=\emptyset$. 
\end{lem}
\begin{proof}
Choose a halfspace $\mf{h}\in\mc{H}(T)$ with $x\in\mf{h}$ and $y\in\mf{h}^*$. A point in $(\partial_MG_x,G)\cap(\partial_MG_y,G)$ would be represented by asymptotic Morse rays $r_x,r_y\cu\G$ contained in neighbourhoods of $\G(\mf{h})$ and $\G(\mf{h}^*)$, respectively. This would contradict Corollary~\ref{choose a side}(3), so no such point can exist.
\end{proof}

Corollary~\ref{component cor}, Lemma~\ref{open map} and Corollary~\ref{undistorted vertex groups cor} imply the following, which is Corollary~\ref{cor intro}.

\begin{cor}\label{td cor}
If $(\partial_MG_x,G)$ is totally disconnected for every $x\in T^{(0)}$, then $\partial_MG$ is totally disconnected.
\end{cor}

The following is not required in the proof of the main theorems, but it looks like a useful observation (for instance, when studying whether $\partial_MG$ is $\s$--compact).

\begin{rmk}
Recall that we have defined an ultrafilter $\s(\xi)\cu\mc{H}(T)$ for every point $\xi\in\partial_MG$. This yields a map $f\colon\partial_MG \ra \overline T:=T^{(0)}\sqcup\partial T$, where $\partial T$ is the set of ends of $T$. 

Every halfspace $\mf{h}\in\mc{H}(T)$ can be extended to a subset $\overline{\mf{h}}\cu\overline T$ by adding the ends of $T$ that it contains. This yields a partition $\overline T=\overline{\mf{h}}\sqcup\overline{\mf{h}^*}$. It is customary to endow $\overline T$ with the topology having the collection of sets $\overline{\mf{h}}$ with $\mf{h}\in\mc{H}(T)$ as a subbasis.

Lemma~\ref{clopen partition} implies that $f\colon\partial_MG \ra \overline T$ is continuous. Moreover, for every vertex $x\in T$, the proof of Lemma~\ref{line ultrafilter}(3) shows that $f^{-1}(x)$ is exactly the set of points of $\partial_MG$ that are represented by rays in $\G$ at bounded distance from the stabiliser $G_x$. If $\xi\in\partial T$, it is possible to show that $f^{-1}(\xi)$ is either empty or a singleton, but this requires a bit more work.

We emphasise that the sets $T^{(0)}$ and $\partial T$ are neither open nor closed in $\overline T$, since $T$ will normally be locally infinite in our setting.
\end{rmk}

\begin{rmk}
The attentive reader might have noticed that the majority of the proof of Theorem~\ref{components intro} works more generally when $G$ acts cocompactly on a $\CAT$ cube complex and all hyperplane-stabilisers are finitely generated, undistorted and with trivial relative Morse boundary in $G$. We have chosen to restrict to actions on trees (the main case of interest) in order to avoid additional difficulties in the following two spots.
\begin{enumerate}
\item If $G$ acts essentially and cocompactly on a $\CAT$ cube complex with undistorted hyperplane-stabilisers, it is not clear if vertex-stabilisers will have to be undistorted as well (cf.\ Corollary~\ref{undistorted vertex groups cor}). It is possible that this can be shown along the lines of \cite[Theorem~A]{Groves-Manning-impr}.
\item The last paragraph of the proof of Lemma~\ref{line ultrafilter}(3) would be a bit more delicate in a $\CAT$ cube complex. One would probably need to require stabilisers of all \emph{intersections of pairwise-transverse hyperplanes} to be finitely generated.
\end{enumerate}
\end{rmk}

\section{Proof of Theorem~\ref{embedding intro}.}

In order to prove that the Morse boundary of the vertex group $V$ topologically embeds in the Morse boundary of the ambient group $G$, we will show that uniformly Morse rays $\alpha\cu V$ remain uniformly Morse in $G$ (see Corollary~\ref{uniformly Morse} below).

One possible approach to this would consider a quasi-geodesic $\g\cu G$ with endpoints on $\alpha$ and use this to construct a quasi-geodesic $\g'\cu V$ with the same endpoints and still getting roughly as far from $\alpha$ as $\g$. Lemma~\ref{path decomposition} allows us to decompose $\g$ as a concatenation $\g_1\g_2\dots\g_m$, where the endpoints of each $\g_i$ are near a coset of an edge group $v_iE_{j_i}\cu V$. We might thus hope to replace each $\g_i$ with a quasi-geodesic $\g_i'\cu v_iE_{j_i}$ with nearby endpoints, in order to form a quasi-geodesic $\g'\cu V$. Since $\alpha$ is Morse in $V$, the quasi-geodesic $\g'$ would then stay close to $\alpha$ and, if edge groups have trivial relative Morse boundary in $V$, the $\g_i'$ would all have to be quite short. Hence the $\g_i$ would also be short and $\g$ would stay close to $\g'$ and $\alpha$, showing that $\alpha$ is Morse in $G$.

The problem with this strategy is that, even if $\g$ and the $\g_i'$ are quasi-geodesics, there is no guarantee that the concatenation $\g'=\g_1'\g_2'\dots\g_m'$ will be a quasi-geodesic. Instead, we will follow a different approach based on divergence, which is the content of Proposition~\ref{uniform divergence}.

Assuming that $\alpha$ has ``slow'' divergence function in $G$, we will show that it also has slow divergence function in $V$ (which contradicts the assumption that $\alpha$ is Morse in $V$). The proof proceeds along similar lines to the above sketch. If $\g\cu G$ were a short path with endpoints on $\alpha$ avoiding a large ball, then we could decompose $\g$ as $\g_1\g_2\dots\g_m$, and replace each $\g_i$ with a short path in $V$, using the assumption that edge groups are relatively wide in $V$ (in the form of Lemma~\ref{rel as cut pts lem}). The new path $\g'$ still avoids a ball of roughly the same radius. Here, the advantage is that it does not matter whether or not $\g'$ is a quasi-geodesic. Of course, asking that edge groups be relatively wide in $V$ might be stronger than simply asking that edge groups have trivial relative Morse boundary in $V$.

We now give precise proofs, considering the following setting.

\begin{ass}\label{ass2}
Let $G$ be a finitely generated group with a non-elliptic, minimal action without inversions on a simplicial tree $G\acts T$. Fix a base vertex $p\in T$ with stabiliser $V$. We assume that the stabiliser of every edge $\mf{e}\cu T$ incident to $p$ is finitely generated and undistorted in $G$.

Let $D_p$ and $E_1,\dots,E_k\leq V$ be as in Lemma~\ref{path decomposition}. Choose finite generating sets $S_{E_i}\cu E_i$, $S_{V}\cu V$ and $S_G\cu G$, ensuring that:
\[S_{E_1}\cup\dots\cup S_{E_k}\cu S_{V}\cu S_G.\] 
Let $\Theta_i\cu\Delta\cu\G$ be the corresponding Cayley graphs of $E_i$, $V$ and $G$, respectively. We denote their intrinsic path metrics by $d_{\Theta_i}$, $d_{\Delta}$ and $d_{\G}$. 

The inclusions $(\Theta_i,d_{\Theta_i}) \hookrightarrow (\Delta,d_{\Delta}) \hookrightarrow (\G,d_{\G})$ are all $1$--Lipschitz. In view of Lemma~\ref{undistorted vertex groups lem}, we can fix a constant $C\geq 1$ such that they are all $C$--bi-Lipschitz.

We assume that $E_1,\dots,E_k$ are relatively wide in $V$. Let $K\geq 1$ be the maximum of the constants $K_0$ provided by Lemma~\ref{rel as cut pts lem}(3) in relation to the pairs $E_i\leq V$ with $1\leq i\leq k$ and the Cayley graphs $\Theta_i$ and $\Delta$.
\end{ass}

We write $\delta_{\alpha}^{\Delta}(\cdot,\cdot)$ and $\delta_{\alpha}^{\G}(\cdot,\cdot)$ when we need to specify the graph used to compute divergence.

\begin{prop}\label{uniform divergence}
There exist $\eta\in (0,1)$ and $K'$, both depending only on $C,D_p,K$, such that, for every geodesic ray $\alpha\cu(\Delta,d_{\Delta})$ and every $r>K'$, we have:
\[\delta_{\alpha}^{\Delta}(r,\eta\cdot\tfrac{1}{16C}) \leq K'\cdot \delta_{\alpha}^{\G}(r,\tfrac{1}{16C}). \]
\end{prop}
\begin{proof}
Let $\alpha\cu\Delta$ be a geodesic ray based at a vertex. Recall that $\alpha$ is $C$--bi-Lipschitz as a path in $(\G,d_{\G})$. Suppose that there exist integers $s,r$ such that the vertices $\alpha(s\pm r)$ can be joined by a path $\g\cu\G$ that avoids the ball $\mc{B}_{\G}(\alpha(s),\tfrac{r}{16C})$. 

Our goal is to construct a path $\g'$ joining $\alpha(s\pm r)$, so that $\g'$ is entirely contained in $\Delta$, avoids the ball $\mc{B}_{\Delta}(\alpha(s),\eta\cdot\tfrac{r}{16C})$, and has length $|\g'|\leq K'\cdot|\g|$. The constants $K'$ and $\eta$ will be determined at the end of the proof. Note that we are only interested in the situation where $r>K'$.

By Lemma~\ref{path decomposition}, we can decompose $\g=\g_1\g_2\dots\g_m$ with the endpoints $\g_i^{\pm}$ in the neighbourhood $\mc{N}_{\G}(v_iE_{j_i},D_p)$ for some $v_i\in V$ and $1\leq j_i\leq k$. Let $\beta_i\cu v_i\Theta_{j_i}$ be a geodesic (for $d_{\Theta_{j_i}}$) whose endpoints $\beta_i^{\pm}$ satisfy $d_{\G}(\beta_i^{\pm},\g_i^{\pm})\leq D_p$. Note that $\beta_i$ is contained in $\Delta$ and is $C$--bi-Lipschitz for $d_{\Delta}$.

Suppose that $r>32C(D_p+K)$. Choose $\eps>0$ small enough that:
\[ \tfrac{1-\eps}{16C}-\eps K\geq \tfrac{1}{32C}. \]

\smallskip
{\bf Claim:} \emph{there exists a path $\beta_i'\cu\Delta$ with the same endpoints as $\beta_i$, such that $\beta_i'$ avoids the ball $B:=\mc{B}_{\Delta}(\alpha(s),\eps\cdot\tfrac{r}{16C})$ and has length $|\beta_i'|\leq K|\beta_i|$.}

\smallskip\noindent
\emph{Proof of claim.}
If $\beta_i$ is disjoint from $B$, we can simply take $\beta_i'=\beta_i$. 

Otherwise, choose a vertex $y\in\beta_i\cap B$ and set $B':=\mc{B}_{\Delta}(y, 2\eps\cdot\tfrac{r}{16C})$. Observe that $B\cu B'$. Thus, it suffices to construct the path $\beta_i'$ so that it avoids the ball $B'$.

Parametrise $\beta_i$ by arc-length so that $\beta_i(0)=y$. Suppose first that the domain of $\beta_i$ contains an interval $[-\rho,\rho]$ satisfying the following three inequalities:
\begin{align*}
\rho&>K, & \rho&>\tfrac{\eps Kr}{8C}, & \rho&>\tfrac{\eps r}{8}.
\end{align*}

Since $\rho>K$ and $\beta_i$ is contained in a left coset of $E_{j_i}$ in $V$, we can apply Lemma~\ref{rel as cut pts lem}, which provides a path $\beta_i'\cu\Delta$ connecting the points $\beta_i(\pm\rho)$, avoiding $\mc{B}_{\Delta}(y,\tfrac{\rho}{K})$ and with length $|\beta_i'|\leq K\rho$. By the second inequality, this path avoids the ball $B'$, which is contained in $\mc{B}_{\Delta}(y,\tfrac{\rho}{K})$. 

We can then prolong $\beta_i'$ along $\beta_i$ until it reaches $\beta_i^{\pm}$. Since $\beta_i$ is $C$--bi-Lipschitz with respect to $d_{\Delta}$, the third inequality ensures that $\beta_i$ can only meet the ball $B'$ at times in the interval $(-\rho,\rho)$. Thus, $\beta_i'$ avoids $B'$ even after prolonging. This is the path required by the claim.

We are left to consider the case when the domain of $\beta_i$ does not contain a sufficiently long interval $[-\rho,\rho]$. Suppose without loss of generality that $d_{\G}(y,\beta_i^-)\leq d_{\G}(y,\beta_i^+)$. Then, recalling that $\beta_i$ is $1$--Lipschitz with respect to $d_{\G}$, we must have:
\begin{align*}
\max\{K, \tfrac{\eps Kr}{8C},\tfrac{\eps r}{8}\} \geq d_{\G}(y,\beta_i^-) \geq d_{\G}(y,\g_i^-)-D_p &\geq d_{\G}(\g_i^-,\alpha(s))-d_{\G}(y,\alpha(s))-D_p \\
&\geq (1-\eps)\tfrac{r}{16C}-D_p.
\end{align*}
The last inequality is due to the fact that $\g_i^-$ lies on $\g$, hence outside $\mc{B}_{\G}(\alpha(s),\tfrac{r}{16C})$, whereas $y$ lies within $B$, hence also in the larger ball $\mc{B}_{\G}(\alpha(s),\eps\cdot\tfrac{r}{16C})$. 

Given our choice of $\eps$, and recalling that $r>32C(D_p+K)$, we obtain a contradiction:
\[ (1-\eps)\tfrac{r}{16C}-D_p\geq [\tfrac{1}{32C}+\eps K]\cdot r-D_p>K+\eps Kr\geq\max\{K, \tfrac{\eps Kr}{8C},\tfrac{\eps r}{8}\} .\]
\hfill$\blacksquare$

\smallskip
Now, we complete the proof of the proposition. Recall that $\beta_i'$ has the same endpoints as $\beta_i$, namely $\beta_i^{\pm}$. Let $\zeta_i\cu\Delta$ be a shortest path connecting the points $\beta_i^+$ and $\beta_{i+1}^-$. Joining the paths $\beta_i'$ provided by the claim by the paths $\zeta_i$, we form a path $\g'$ entirely contained in $\Delta$ and with the same endpoints as $\g$.

Recalling that $\g_i^+=\g_{i+1}^-$ and $d_{\G}(\beta_i^{\pm},\g_i^{\pm})\leq D_p$, we have:
\[d_{\Delta}(\beta_i^+,\beta_{i+1}^-)\leq C\cdot d_{\G}(\beta_i^+,\beta_{i+1}^-)\leq 2CD_p.\]
It follows that $|\zeta_i|\leq 2CD_p$ for each $i$. 

Now, suppose that $r>\tfrac{32C^2D_p}{\eps}$. Then, since each $\beta_i'$ avoids the ball $B=\mc{B}_{\Delta}(\alpha(s),\eps\cdot\tfrac{r}{16C})$, while each $\zeta_i$ has length at most $2CD_p$, the path $\g'$ avoids the ball:
\[\mc{B}_{\Delta}(\alpha(s),\eps\cdot\tfrac{r}{16C}-CD_p)\supseteq\mc{B}_{\Delta}(\alpha(s),\tfrac{\eps}{2}\cdot\tfrac{r}{16C}).\]
Thus, setting $\eta:=\tfrac{\eps}{2}$, the path $\g'$ avoids the ball $\mc{B}_{\Delta}(\alpha(s),\eta\cdot\tfrac{r}{16C})$.

We are left to bound the length of $\g'$:
\begin{align*}
|\g'|\leq \sum |\zeta_i| + \sum |\beta_i'|&\leq m\cdot 2CD_p + K\cdot\sum|\beta_i|=2mCD_p + K\cdot\sum d_{\Theta_{j_i}}(v_i^{-1}\beta_i^-,v_i^{-1}\beta_i^+) \\ 
&\leq 2CD_p\cdot|\g| + KC^2\cdot\sum d_{\G}(\beta_i^-,\beta_i^+) \\
&\leq 2CD_p\cdot|\g| + KC^2\cdot\sum (|\g_i|+2D_p)\leq (2CD_p+KC^2+2KC^2D_p)|\g|.
\end{align*}

In conclusion, taking
\[K':=\max\{32C(D_p+K),\ \tfrac{32C^2D_p}{\eps},\ 2CD_p+KC^2+2KC^2D_p\}\]
and assuming that $r>K'$, we have constructed a path $\g'$ with the same endpoints as $\g$, entirely contained in $\Delta$, avoiding the ball $\mc{B}_{\Delta}(\alpha(s),\eta\cdot\tfrac{r}{16C})$, and with length $|\g'|\leq K'\cdot|\g|$. This shows that
\[\delta_{\alpha}^{\Delta}(r,\eta\cdot\tfrac{1}{16C}) \leq K'\cdot \delta_{\alpha}^{\G}(r,\tfrac{1}{16C}),\]
as required.
\end{proof}

Along with Lemma~\ref{Morse vs divergence}, Proposition~\ref{uniform divergence} has the following consequence.

\begin{cor}\label{uniformly Morse}
For every Morse gauge $N$, there exists a Morse gauge $N'$ such that every $N$--Morse geodesic ray $\alpha\cu\Delta$ is an $N'$--Morse quasi-geodesic in $\G$.
\end{cor}

We are now ready to prove Theorem~\ref{embedding intro}.

\begin{proof}[Proof of Theorem~\ref{embedding intro}]
By Corollary~\ref{undistorted vertex groups cor}, $V$ is undistorted in $G$. Recall that the relative Morse boundary $(\partial_MV,G)$ is by definition a subset of $\partial_MV$. Since all incident edge groups $E\leq V$ are relatively wide in $V$, we are in the setting of Assumption~\ref{ass2}, so Corollary~\ref{uniformly Morse} implies that $(\partial_MV,G)=\partial_MV$. 

In addition, Corollary~\ref{uniformly Morse} guarantees that the natural inclusion $\partial_MV\hookrightarrow\partial_MG$ is a Morse preserving map in the sense of \cite[Definition~4.1]{Cordes}. Thus, the fact that $\partial_MV\hookrightarrow\partial_MG$ is a topological embedding follows from \cite[Proposition~4.2]{Cordes}.
\end{proof}

\bibliography{mybib}
\bibliographystyle{alpha}

\end{document}